\documentclass[a4paper,reqno,11pt]{amsart}
\usepackage{amsfonts,amssymb,verbatim,amsmath,amsthm,latexsym,textcomp,amscd}
\usepackage{latexsym,amsfonts,amssymb,epsfig,verbatim}
\usepackage{amsmath,amsthm,amssymb,latexsym,graphics,textcomp}
\usepackage{graphicx}
\usepackage{color}
\usepackage{url}
\usepackage{enumerate}
\usepackage[mathscr]{euscript}
\usepackage{mathtools} 
\usepackage{tikz}
\usetikzlibrary{cd}
\usepackage{dsfont}
\usetikzlibrary{matrix}
\usepackage{sseq}
\usepackage{hyperref}

\input xy
\xyoption{all}

\theoremstyle{definition}
\newtheorem{theorem}{Theorem}[section]
\newtheorem{prop}[theorem]{Proposition}

\newtheorem{cor}[theorem]{Corollary}
\newtheorem{defn}[theorem]{Definition}

\newtheorem{exam}[theorem]{Example}

\newtheorem{thm}[theorem]{Theorem}

\newtheorem*{theorema}{Theorem A}
\newtheorem*{theoremb}{Theorem B}
\newtheorem*{theoremc}{Theorem C}

\usepackage{tipa}

\newcommand{\C}{{\mathbb C}}

\newcommand{\Q}{{\mathbb Q}}

\newcommand{\Spec}{\mbox{Sp}}

\newcommand\FF{{\mathcal F}}

\newcommand\LL{{\mathcal L}}
\newcommand\MM{{\mathcal M}}

\newcommand\PP{{\mathcal P}}

\newcommand\PMF{{\PP\kern-2pt\MM\FF}}

\newcommand\PML{{\PP\kern-2pt\MM\LL}}

\newcommand{\T}{{\mathbb T}}

\newcommand{\fsubd}{\mathrel{{\scriptstyle\searrow}\kern-1ex^d\kern0.5ex}}
\newcommand{\bsubd}{\mathrel{{\scriptstyle\swarrow}\kern-1.6ex^d\kern0.8ex}}
\newcommand{\fsubeq}{\mathrel{\raise-.7ex\hbox{$\overset{\searrow}{=}$}}}
\newcommand{\bsubeq}{\mathrel{\raise-.7ex\hbox{$\overset{\swarrow}{=}$}}}

\newcommand{\tsh}[1]{\left\{\kern-.9ex\left\{#1\right\}\kern-.9ex\right\}}

\numberwithin{equation}{section}
\newenvironment{myeq}[1][]
{\stepcounter{theorem}\begin{equation}\tag{\thetheorem}{#1}}
	{\end{equation}}
\newenvironment{meq*}[1][]
{\stepcounter{theorem}\begin{equation*}\tag{\thetheorem}{#1}}
	{\end{equation*}}
\newtheorem{subsec}[theorem]{}

\begin{document}

\title[Equivariant rational projective spaces]{Rational stable homotopy type of equivariant projective spaces and Grassmannians}

\author{Samik Basu}
\address{Stat-Math Unit\\ Indian Statistical Institute\\ Kolkata 700108\\ India}
\email{samik.basu2@gmail.com}
\author{Vanny Doem}
\address{Math Department\\ Indian Institute of Technology Kharagpur\\ Kharagpur 721302\\ India}
\email{vanny.doem@gmail.com}
\author{Chandal Nahak}
\address{Math Department\\ Indian Institute of Technology Kharagpur\\ Kharagpur 721302\\ India}
\email{cnahak@maths.iitkgp.ac.in}

\date{\today}

\maketitle

\begin{abstract}
\leftskip=0cm \rightskip=0cm
We prove explicit rational stable splittings of equivariant complex projective spaces $\C P(V)$ and Grassmannians $Gr_n(V)$, for complex representations $V$. When $V$ is a sum of one-dimensional representations, both $\C P(V)$ and $Gr_n(V)$ are rationally a wedge of representation spheres. For general finite groups $G$ and $V$  a sum of irreducible representations which are not necessarily one-dimensional, we show that $\C P(V)$ splits rationally as a wedge of Thom spaces over irreducible factors in $V$. For $Gr_n(V)$, the factors in the corresponding rational splitting are a smash product of Thom spaces over lower Grassmannians on irreducible factors in $V$.
\end{abstract}


\section {Introduction}

In equivariant stable homotopy theory, one of the most useful splitting results is the tom Dieck splitting \cite[Satz 2]{tD75}, \cite[\S V.11]{LMSM86}: The fixed point spectrum of a suspension $G$-spectrum is a wedge sum of the $W_GH$-Borel construction on $H$-fixed point spectra for conjugacy classes of subgroups $H$. In the realm of rational equivariant stable homotopy theory, the so-called Greenlees-May splitting \cite[A.1]{GM95} states that: The rational $G$-spectra for a finite group $G$ is equivalent to the wedge sum of Eilenberg-MacLane spectra on its equivariant homotopy groups. In this paper, we encounter a different sort of rational splitting for complex projective spaces and Grassmannians.  

In the non-equivariant setting, it is well-known that stably, the complex projective space $\C P^n$ splits rationally into a wedge of spheres 
$$\Sigma^\infty\C P^n_+\simeq_\Q\bigvee_{i=0}^{n}S^{2i}.$$ 
This follows from the fact that its rational homology is concentrated in even degrees, a property that forces the connecting maps in the associated cofiber sequences to be zero. 

Now we wish to examine this result in the equivariant setting. Given a complex representation $V$ of a finite group $G$, one defines the equivariant complex projective space $\C P(V)$ to be the space of complex lines in $V$, for which its underlying space is $\C P^{\text{dim}_\C (V)-1}$. More generally one defines the equivariant Grassmannians $Gr_n(V)$ to be the set of all $n$-dimensional complex linear subspaces $V$, so that $Gr_1(V)=\C P(V)$. The main purpose of the present paper is to explicitly write down rational stable splittings of $\C P(V)$ and $Gr_n(V)$, which come up while building these spaces as cell complexes. 

We now describe the results in this paper. The simplest case is when the complex representation $V$ is a direct sum of one dimensional representations. This is always the case when $G$ is an Abelian group. 
We write (where $\phi_i$ are one-dimensional) 
$$V= \sum_{i=1}^{n}\phi_i, \quad \omega_i= (\phi_1 + \cdots + \phi_i)\otimes \phi_{i+1}^{-1},$$
and for a Schubert symbol $\sigma = (\sigma_1, \cdots, \sigma_n)$, 
$$ W_\sigma= \sum_{i=1}^n W_{\sigma_i}, \quad W_{\sigma_i}= (\sum_{j=1}^{\sigma_i-1} \phi_j - \sum_{l=1}^{i-1} \phi_{\sigma_l})\otimes \phi_{\sigma_i}^{-1}.$$ 
In terms of this notation, we have (Theorem \ref{splitting CP(V)} and Theorem \ref{thm:split Gr_n(V)})

\begin{theorema} Suppose $V$ is a sum of one-dimensional representations denoted as above. Then, there are equivalences 
$$\Sigma^\infty\C P(V)_+\simeq_\Q \bigvee_{i=0}^{n-1}S^{\omega_i},$$
$$ \Sigma^\infty Gr_n(V)_+ \simeq_\Q \bigvee_{\substack{\mbox{\tiny Schubert}\\ \mbox{\tiny symbols } \sigma}} S^{W_\sigma} .$$
\end{theorema}

The proof of Theorem A involves writing $\C P(V)$ and $Gr_n(V)$ as equivariant cell complexes, and noting that the attaching maps are rationally trivial. Of course, the expressions rely on the order in which we write the one-dimensional representations $\phi_i$.

For general representations $V$, express $V$ as a direct sum of irreducible representations: $V= \sum \psi_i$. The cofiber of $\C P(V) \to \C P(V+\psi)$ may be identified as a Thom space \cite{Ati61}. 
We show that the resulting filtration induces a rational stable splitting of the equivariant projective space $\C P(V)$ into wedges of Thom spaces. (Theorem \ref{thm:split thom space})
\begin{theoremb} Let 
\[ V= \sum_{i=1}^n \psi_i,\quad V_i = \sum_{j\leq i} \psi_j,\]
where $\psi_j$ are irreducible representations. There is an equivalence 
$$\Sigma^\infty\C P(V)_+\simeq_\Q \bigvee_{i=0}^{n-1}\text{Th}(V_i\otimes\gamma_{\psi_{i+1}}),$$
where $\gamma_{\psi_j}$ is the canonical line bundle over $\C P(\psi_j)$.
\end{theoremb} 

We carry this argument forward to the case of $Gr_n(V)$. In this case we demonstrate what happens in a single step where $V_0 \subset V$ and $V/V_0 = V'$. Rationally, this splits as a wedge of lower Grassmannians on $V_0$ smashed with Thom spaces of lower Grassmannians on $V'$. (Theorem \ref{thm:general split gr}) 
\begin{theoremc}
 There is a stable homotopy equivalence
    $$\Sigma^\infty Gr_n(V)_+\simeq_\Q\bigvee_{k=1}^{n}\text{Th}(Gr_k(V'), V_0\otimes\xi_k)\land Gr_{n-k}(V_0)_+\bigvee Gr_n(V)_{V_0 +}.$$
\end{theoremc} 
One may iterate the equivalence in this theorem to further reduce the expression to Grassmannians over irreducible representations. 

\subsection{Organization}
In \S \ref{prelim}, we recall some preliminaries that are used throughout the paper.  The \S \ref{sec:cp(v)} deals with the case of Abelian groups proving Theorem A, and \S \ref{sec:cp(v)gen} deals with the proof of Theorem B. Finally we prove Theorem C in \S \ref{sec:gr_n(v)}. 

\subsection{Notation} Throughout this paper, we use the following notations. 
\begin{itemize}
\item $G$ denotes a finite group. 
\item It is assumed that the underlying category is the $G$-equivariant rational stable homotopy category. 
\item $W_GH=N_GH/H$ denotes the Weyl group of a subgroup $H$.
\item The notation $\C P(V)$ is the equivariant complex projective space of lines in $V$.
\item The notation $Gr_n(V)$ stands for the Grassmannian of $n$-planes in $V$. 
\item For a $G$-bundle $E$, let $\text{Th}(E)$ be its Thom space. 
\item  $\mathbb{T}$ stands for the circle group $S^1$, and its representation corresponding to the action of $\mathbb{T}$ on $\C$ by multiplication is denoted as $\lambda$.
\end{itemize}

%

\noindent
{\bf Acknowledgments.}
The second author was supported by the Indian Council for Cultural Relations (ICCR), Research Grant (No.PHN/ICCR/321/1/2021), India.

\section{Preliminaries}\label{prelim}
In this section we review some standard constructions to be used later. Throughout the paper, we restrict our attention to finite groups $G$.  
\subsection{Equivariant projective spaces}
Let $G$ be a finite group and $V$ a representation of $G$ endowed with a $G$-invariant inner product. We write $D(V)$ for the unit disk in $V$ and $S(V)$ for the unit sphere of $V$, respectively expressed as, 
$$D(V)=\{v\in V~|~ \langle v,v\rangle\leq 1\} \text{ and } S(V)=\{v\in V~|~ \langle v,v\rangle=1\}.$$
Denote by $S^V$ the one-point compactification of $V$ with basepoint at $\infty$, and observe that $S^V\cong D(V)/S(V)$.

Write $\C P(V)$ for the equivariant complex projective space with linear $G$-action associated to a representation $V$. More precisely we make the following definition. 
\begin{defn}\label{def:CP(V)}
    For a complex representation $V$ of a finite group $G$, let $\C P(V)$ be the set of one-dimensional complex linear subspaces of $V$.
\end{defn}
We can topologize $\C P(V)$ as a quotient $G$-space
$$\C P(V)\cong(V\smallsetminus\{0\})/\C^*$$
where $\C^*=\C\smallsetminus\{0\}$ acting on $V\smallsetminus\{0\}$ by scalar multiplication. If we write $\T$ for the circle group and $\lambda$ for the one-dimensional representation of $\T$, then Definition \ref{def:CP(V)} can be rewritten as follows \cite[\S2]{CGK00}. 

 \begin{prop}\label{def2:CP(V)}
    The space $\C P(V)$ is $G$-homeomorphic to the quotient $G$-space
    $$\C P(V)\cong S(V\otimes\lambda)/\T$$
    where $\lambda$ is the one-dimensional representation of $\T$, given as $\T$ acting on $\C$ by multiplication.
 \end{prop}
Following a commutative triangle
\[\begin{tikzcd}
	{S(V\otimes\lambda)} & {\mathbb{C}P(V)} \\
	{S(V\otimes\lambda)/\mathbb{T}}
	\arrow["q", from=1-1, to=1-2]
	\arrow["f"', from=1-1, to=2-1]
	\arrow["{\tilde{q}}"', from=2-1, to=1-2]
\end{tikzcd}\]
such that map $q: S(V\otimes\lambda)\to\C P(V)$, is described as $(v, z)\mapsto \text{span}\{v\}$, it is evident that $\C P(V)\cong S(V\otimes\lambda)/\T$.

\subsection{Equivariant cell structure of $\C P(V)$}
We recall an equivariant cell structure of $\C P(V)$  constructed in \cite[\S3]{Lew88}. The usual CW complex structure on projective spaces yields the equivariant cell structure where cells are of the form $D(W)$ attached along maps out of $S(W)$, in the case $V$ is a direct sum of one-dimensional representations. 
 
  We have an explicit description of the $G$-cell structure of $\C P(V)$ as follows. Write $V=V'\oplus\phi$ for one-dimensional $\phi$, we have a cofiber sequence
 \begin{myeq}\label{eq:cof v, v'}
     \C P(V')_+\to\C P(V)_+\to S^{V'\otimes\phi^{-1}}.
 \end{myeq}The $G$-cells are of the form $D(V'\otimes\phi^{-1})$. Now suppose $V$ is a representation with $\text{dim}V=n$, and write
 \begin{myeq} \label{sum of irred}
     V=\sum_{i=1}^n\phi_i \text{ and } V_i=\sum_{j=1}^i\phi_j
 \end{myeq}for one-dimensional representations $\phi_i$. 
 
 Following \cite[\S3]{CGK00} we choose the complete flag of $V_i$ with respect to the above $V_i$ \eqref{sum of irred}
 $$0=V_0\subset V_1\subset V_2\subset \cdots\subset V_n=V.$$
Associated with this flag we obtain a complete cellular filtration of $\C P(V)$
\begin{myeq} \label{eq:filter cp(v)}
  \emptyset\subset\C P(V_1)\subset\C P(V_2)\subset\cdots\subset\C P(V_n)=\C P(V)  
\end{myeq}with a cofiber sequence \eqref{eq:cof v, v'}
\begin{myeq} \label{cof seq}
  \C P(V_i)_+\to\C P(V_{i+1})_+\to S^{W_i}  
\end{myeq} where $W_i =V_i\otimes\phi_{i+1}^{-1}$, and $\phi_i$ are one-dimensional representations.

Notice that for one-dimensional $\phi$, the space $\C P(\phi)$ is just a point, and we have $$\C P(V)\cong\C P(V\otimes\phi).$$

\subsection{Thom complexes and the filtration of $\C P(V)$}\label{subsec:cof thom space}
If $G$ is Abelian, every complex representation is a direct sum of one-dimensional representations. This is not true for general finite groups. In this case we obtain a filtration where cofibers are described using Thom spaces. 

Let $G$ be a finite group and $V$ a representation of $G$. Write $V=V'+\psi$ for an irreducible representation $\psi$. From Definition \ref{def2:CP(V)} we obtain a cofiber sequence via the mapping cone
 \begin{myeq}\label{eq2:cof v, v'}
     S(V'\otimes\lambda)/\T\xhookrightarrow{i}S(V\otimes\lambda)/\T\to C(i)
 \end{myeq}where $C(i)$ is the subquotient $S(V\otimes\lambda)/\T$ by $S(V'\otimes\lambda)/\T$.

 From \cite{Ati61} we have that $\C P(V)\setminus \C P(V')$ is a vector bundle over $\C P(\psi)$. More formally
 $$\C P(V)\setminus\C P(V')\cong D(V'\otimes\gamma_\psi\to\C P(\psi))$$
 where $\gamma_\psi$ is the canonical line bundle over $\C P(\psi)$ described as 
 $$S(\psi)\times_\T\C\xrightarrow[]{\gamma_\psi}\C P(\psi), (v,z)\mapsto\text{span}\{v\}.$$
We write this Thom space as $\text{Th}(V'\otimes\gamma_\psi)$. Then the cofiber sequence \eqref{eq2:cof v, v'} simplifies to
 \begin{myeq}\label{eq:thom cof seq}
     \C P(V')_+\to\C P(V)_+\to \text{Th}(V'\otimes\gamma_\psi).
 \end{myeq}

\subsection{Equivariant Grassmannians}
Let $G$ be a finite group and $V$ an $m$-dimensional complex $G$-representation endowed with a $G$-invariant inner product. For $n\leq m$ we write $Gr_n(V)$ for the complex Grassmannian with the linear $G$-action associated to  $V$.
\begin{defn}\label{def:gr_n(v)}
    Let $Gr_n(V)$ be the set of all $n$-dimensional complex linear subspaces $V$, topologized as usual as the quotient of the Stiefel manifold on $n$ orthogonal vectors in $V$.
\end{defn}

Notice that $Gr_1(V)\cong \C P(V)$ is the equivariant complex projective space.

\subsection{Equivariant Schubert cell structure of $\text{Gr}_n(V)$} 
If $G$ is an Abelian group, then $V$ is a direct sum of one-dimensional representations, and the usual filtration of the $Gr_n(V)$ by Schubert cells yields an equivariant cell structure. We recall this below.

Assume that $V$ is a representation such that $\text{dim}V=m$. Following \cite[\S6]{MS74} we choose a complete flag
 \begin{myeq}\label{eq:filt v_i for gr_n(v)}
     0=V_0\subset V_1\subset V_2\subset \cdots\subset V_m=V.
 \end{myeq}and write $\phi_i = V_i/V_{i-1}$. Recall a Schubert symbol $\sigma=(\sigma_1,\sigma_2,\cdots,\sigma_n)$ where $\sigma_i$ are integers and satisfy
 \begin{myeq}\label{eq: schubert index}
    1\leq\sigma_1<\sigma_2<\cdots<\sigma_n\leq m.
 \end{myeq}Rewriting \eqref{eq:filt v_i for gr_n(v)} by using \eqref{eq: schubert index}
 \begin{myeq}\label{eq:filt v sigma gr}
     0\subset V_{\sigma_1}\subset V_{\sigma_2}\subset \cdots \subset V_{\sigma_n}
 \end{myeq}so that $V_{\sigma_i}=\phi_{\sigma_1}+\phi_{\sigma_2}+\cdots+\phi_{\sigma_i}$.
 
 Following \eqref{eq:filt v sigma gr} we write a non-equivariant Schubert cell
 \begin{myeq}\label{eq:schubert cell}
    e(\sigma)=\{W\in Gr_n(V) ~|~ \text{dim}(W\cap V_{\sigma_i})=i \text{ and } \text{dim}(W\cap V_{\sigma_i-1})=i-1\}. 
 \end{myeq}Such an $n$-plane $W$ consists of a basis $z_1, z_2,\cdots, z_n$ where $z_i\in V_{\sigma_i}$ and non-zero in $V_{\sigma_i}/V_{\sigma_{i}-1}$. In terms of matrices this $n$-plane $W$ is represented as an $m\times n$-matrix with rows $z_1, z_2,\cdots, z_n$, and columns indexed by the modified sequence $(\phi_1, \phi_2,\cdots, \phi_m)$. Dividing every $z_i$ by its $\phi_{\sigma_i}$ coordinate, we may assume the last entry in each row is positive, say equal to $1$, and all subsequent entries are zeros, by subtracting a suitable multiple of $z_i$ from the other rows, as we may assume the matrix is in row-reduced echelon form.

As noted in \cite[\S3]{CGK02}, write $V$ as in \eqref{sum of irred} so that the above description gives an equivariant cell structure on $Gr_n(V)$. Indeed since all $V_i$ are $G$-invariant, the cells $V(\sigma)$ are $G$-subspaces. As a $G$-space, this equivariant Schubert cell $V(\sigma)$ is a disk homeomorphic to a product of $(V_{\sigma_i-1}- \sum_{j=1}^{i-1} \phi_{\sigma_j})\otimes\phi_{\sigma_i}^{-1}$. More explicitly using \eqref{eq: schubert index} and rewriting \eqref{eq:schubert cell} for $V(\sigma)$
\begin{myeq}\label{eq:equiv schubert cell}
    V(\sigma)=\{W\in Gr_n(V) ~|~ \text{dim}(W\cap V_{\sigma_i})=i \text{ and } \text{dim}(W\cap V_{\sigma_i-1})=i-1\},
\end{myeq}the equivariant Schubert cells of $Gr_n(V)$ are
\begin{myeq}\label{eq: schubert cell}
    V(\sigma)\cong D(W_\sigma), W_\sigma=\bigoplus_{i=1}^{n}W_{\sigma_i}
\end{myeq}
such that 
$$W_{\sigma_i}=(V_{\sigma_i-1}- \sum_{j=1}^{i-1} \phi_{\sigma_j})\otimes\phi_{\sigma_i}^{-1}.$$
 As noticed in \cite[6.3]{MS74} this cell $V(\sigma)$ is essentially an open cell of dimension 
\begin{myeq}\label{eq:dim sigma}
    |\sigma|=\sum_{i=1}^{n}\sigma_i-i.
\end{myeq}

\subsection{Rational stable homotopy}\label{subsec:rsh}
Let $G$ be a finite group. The central technique we use throughout our article is grounded in a complete algebraic model of rational $G$-spectra in \cite[6.2]{G06}: An equivalence of homotopy categories
$$\Spec_\Q^G\simeq\prod_{H\leq G}D(\Q[W_GH])$$
as triangulated categories. From this algebraic model we have the identification \cite[\S2]{GQ23}
\begin{myeq}\label{eq:id alg model}
    [X,Y]^G\xrightarrow{\cong}\prod_n\prod_{H\leq G}\text{Hom}_{\Q[W_GH]}(\pi_n(\Phi^H(X)),\pi_n(\Phi^H(Y))).
\end{myeq}

\section{Rational stable homotopy type of $Gr_n(V)$ over Abelian groups}\label{sec:cp(v)}
For Abelian groups $G$, every complex representation is a direct sum of one-dimensional representations. Following this, we write a complex representation $V$ of dimension $n$ as 
    $$V=\sum_{i=1}^{n}\phi_i$$
such that $\text{dim}\phi_i=1$. We first show that the complex projective space $\C P(V)$ rationally splits as a wedge of spheres. This expression as a wedge depends on the order in which we arrange the $\phi_i$. More precisely we fix the ordered sequence
\begin{myeq}\label{eq:ordered phi}
   (\phi_1, \phi_2,\cdots,\phi_n) 
\end{myeq} for $\phi_i$ are the elements in $V$, and rewrite \eqref{sum of irred}
\begin{myeq} \label{order one rep} 
V=\sum_{i=1}^n\phi_i  \text{ and } V_i=\sum_{j=1}^{i}\phi_i.
\end{myeq}We have the proper cellular filtration of $\C P(V)$ \eqref{eq:filter cp(v)}
$$\emptyset =\C P(V_0)\subset\C P(V_1)\subset\C P(V_2)\subset\cdots\subset\C P(V_n)=\C P(V)$$
with the cofiber sequences
\begin{myeq} \label{cof seq via thom space}
    \C P(V_{i})_+\to\C P(V_{i+1})_+\to S^{\omega_i}
\end{myeq}where representations
\begin{myeq}\label{eq:rep omega cp}
    \omega_i=V_i\otimes\phi_{i+1}^{-1}
\end{myeq}for some $1\leq i\leq n-1$ and $\phi_i$ are one-dimensional.
 Notice that since $\C P(V_1)=S^0$, the expression $S^{\omega_0}= S^0$.

 We show that this cofiber sequence \eqref{cof seq via thom space} splits in $\Spec_\Q^G$ and yields a rational stable splitting of $\C P(V)$ by using the approach in \S \ref{subsec:rsh}. 
\begin{theorem} \label{splitting CP(V)}
    There is an equivalence
    $$\Sigma^\infty\C P(V)_+\simeq_\Q \bigvee_{i=0}^{n-1}S^{\omega_i}$$
    where $V$ is defined as in \eqref{order one rep} and $\omega_i$ in \eqref{eq:rep omega cp}.
\end{theorem}
\begin{proof}
We proceed by induction on $n=\text{dim}V$. Use the notation \eqref{order one rep} and consider the cofiber sequence \eqref{cof seq via thom space}
$$\C P(V_{n-1})_+\to\C P(V)_+\to S^{\omega_{n-1}}.$$
Inductively we assume that
 $$\C P(V_i)_+\simeq_\Q\C P(V_{i-1})_+\vee S^{\omega_{i-1}}.$$
It is enough to show that the connecting map in the sequence, that is,
 $$S^{\omega_{n}}\xrightarrow[]{\delta_n}\Sigma\C P(V_{n-1})_+$$
 is the zero map. 
 
 By the induction hypothesis
 $$\C P(V_{n-1})_+\simeq_\Q \C P(V_{n-2})_+\vee S^{\omega_{n-2}}\simeq_\Q\bigvee_{i=0}^{n-2}S^{\omega_i}.$$
 We then proceed for $i=n$ as follows. The map
 $$S^{\omega_{n}}\xrightarrow[]{\delta_n}\Sigma\C P(V_{n-1})_+\simeq_\Q\bigvee_{i=0}^{n-2}S^{\omega_i+1}$$ 
 up to homotopy, is a collection of $G$-maps \eqref{eq:id alg model}
 \begin{align*}
 \bigvee_{i=0}^{n-2}[S^{\omega_{n}}, S^{\omega_i+1}]^G
 &\cong\bigoplus_{i=0}^{n-2}\prod_k\prod_{H\leq G}\text{Hom}_{\Q[W_GH]}(\pi_k(\Phi^HS^{\omega_{n}}),  \pi_k(\Phi^HS^{\omega_i+1}))\\
 &\cong\bigoplus_{i=0}^{n-2}\prod_k\prod_{H\leq G}\text{Hom}_{\Q[W_GH]}(\pi_k(S^{\omega_{n}^H}),  \pi_k(S^{\omega_i^H+1}))
 \end{align*}
which gives rise to the zero element since the $H$-fixed point dimensions $|\omega_{n}^H|$ are even while $|\omega_i^H+1|$ are odd for all subgroups $H\leq G$.

Thus the short exact sequence splits
$$\C P(V)_+\simeq_\Q\C P(V_{n-1})_+\vee S^{\omega_{n-1}}.$$
This proves the claim as required.
\end{proof}

In the context of Abelian groups it is not hard to generalize this Theorem \ref{splitting CP(V)} for $Gr_n(V)$. Now suppose $V$ be a representation with $\text{dim}V=m$, and rewrite \eqref{sum of irred}
 \begin{myeq} \label{eq:sum irred gr}
     V=\sum_{i=1}^m\phi_i \text{ and } V_{i}=\sum_{j=1}^i\phi_{j}.
 \end{myeq}Recalling that with a Schubert symbool $\sigma=(\sigma_1,\sigma_2,\cdots,\sigma_n)$ in \eqref{eq: schubert index}, write \eqref{eq: schubert cell} 
\begin{myeq}\label{eq:notation disk gr}
    W_\sigma=\bigoplus_{i=1}^{n}W_{\sigma_i} \text{ and } W_{\sigma_i}=(V_{\sigma_i-1}- \sum_{j=1}^{i-1} \phi_{\sigma_j})\otimes\phi_{\sigma_i}^{-1}.
\end{myeq}

As the complex dimension of $Gr_n(V)$ equals to $d=n(\text{dim}V-n)$, using \eqref{eq: schubert cell} and \eqref{eq:dim sigma} we express
 $$Gr_n(V)=\bigcup_{i=0}^{d}Gr_n(V)^{(i)}$$
 with the property that 
 $$Gr_n(V)^{(i)}=\bigcup_{|\sigma|\leq i}V(\sigma) \text{ and } |\sigma|=\sum_{i=1}^{n}\sigma_i-i.$$
 Then the $G$-CW complex is a sequential colimit of $Gr_n(V)^{(i-1)}$ where $Gr_n(V)^{(i)}$ is a pushout
\[\begin{tikzcd}
	{\bigsqcup\limits_{|\sigma|=i}S(W_\sigma)} & {Gr_n(V)^{(i-1)}} \\
	{\bigsqcup\limits_{|\sigma|=i}D(W_\sigma)} & {Gr_n(V)^{(i)}}
	\arrow[from=1-1, to=1-2]
	\arrow[from=1-1, to=2-1]
	\arrow[from=1-2, to=2-2]
	\arrow[from=2-1, to=2-2]
\end{tikzcd}\]
Following this we have a cofiber sequence
\begin{myeq}\label{eq:cof seq ab}
    Gr_n(V)^{(i-1)}_+\to Gr_n(V)^{(i)}_+\to \bigvee_{|\sigma|=i}S^{W_\sigma}.
\end{myeq}

In the same fashion as in Theorem \ref{splitting CP(V)} we claim that this sequence \eqref{eq:cof seq ab} splits in $\Spec_\Q^G$ and makes an explicit rational stable splitting of $Gr_n(V)$. 
\begin{theorem} \label{thm:split Gr_n(V)}
    There is an equivalence
    $$\Sigma^\infty Gr_n(V)_+\simeq_\Q \bigvee_{\sigma}S^{W_\sigma}$$
    with notations $V$ in \eqref{eq:sum irred gr} and $W_\sigma$ in \eqref{eq:notation disk gr}.
\end{theorem}
\begin{proof}
We show by induction on $i$ that the cofiber sequence \eqref{eq:cof seq ab}
$$Gr_n(V)^{(i-1)}_+\to Gr_n(V)^{(i)}_+\to \bigvee_{|\sigma|=i}S^{W_\sigma}$$
splits rationally. By the induction step
 $$Gr_n( V)^{(i-1)}_+\simeq_\Q Gr_n( V)^{(i-2)}_+\bigvee\bigvee_{|\sigma|=i-1}S^{W_\sigma}\simeq_\Q\bigvee_{|\sigma|\leq i-1}S^{W_\sigma}.$$
 
 We then show for $i$ by observing the connecting map in \eqref{eq:cof seq ab}
 $$\bigvee_{|\sigma|=i}S^{W_\sigma}\xrightarrow[]{\delta_i}\Sigma Gr_n(V)^{(i-1)}_+\simeq_\Q \bigvee_{|\sigma|\leq i-1}S^{W_\sigma+1}$$
 is a zero map.
 For all subgroups $H\leq G$, the fixed points $|W_\sigma^H|$ are even-dimensional and $|W_\sigma^H+1|$ are odd-dimensional. Therefore the map $\delta_i=0$ by \eqref{eq:id alg model}, and this proves the theorem by induction.
\end{proof}

\section{Rational stable homotopy type of $\C P(V)$ over general groups}\label{sec:cp(v)gen}
For the general finite groups $G$, it is no longer true that every complex $G$-representation is a direct sum of one-dimensional representations. In this case, we write $V$ as a direct sum of the irreducible representations of $G$. We then show that the complex projective space $\C P(V)$ rationally splits as a wedge of Thom spaces over some lower complex projective spaces (Theorem \ref{thm:split thom space}). First we recall some related constructions in \S\ref{subsec:cof thom space}, and compute the $H$-fixed points of $\C P(V)$.

 Given a complex $n$-dimensional representation $V$ and the irreducible representations $\psi_i$, we rewrite \eqref{sum of irred}
\begin{myeq}\label{eq:sum gen irred}
    V=\sum_{i=1}^{n}\psi_i \text{ and } V_i=\sum_{j=1}^{i}\psi_j.
\end{myeq}Following \eqref{eq2:cof v, v'} we may resemble \eqref{eq:thom cof seq} to get a cofiber sequence 
\begin{myeq}\label{eq:thom cof seq 2}
     \C P(V_{i})_+\to\C P(V_{i+1})_+\to \text{Th}(V_{i}\otimes\gamma_{\psi_{i+1}})
 \end{myeq}where the canonical line bundle $\gamma_{\psi_{i+1}}$ is given by
 \begin{myeq}\label{eq:line bundle gamma psi}
     S(\psi_{i+1})\times_\T\C\xrightarrow[]{\gamma_{\psi_{i+1}}}\C P(\psi_{i+1}), (v,z)\mapsto\text{span}\{v\}.
 \end{myeq}

For every subgroup $H\leq G$, let $i_H^*:\text{Rep}(G)\to\text{Rep}(H)$ be the restriction. We write
\begin{myeq}\label{eq:sum irred res}
    i_H^*V=\sum_{i=1}^{k}n_i\phi_i\oplus\sum_{j=1}^{l}m_j\alpha_j
\end{myeq}where all $\phi_i$ are one-dimensional $H$-representations, $\alpha_j$ are irreducible $H$-representations with $\dim(\alpha_j)>1$, and the multiplicities $n_i, m_j>0$. Under the restriction, we observe that the $H$-invariant complex lines are basically the one-dimensional $H$-subrepresentations of $i_H^*V$, which are contained in the first term of \eqref{eq:sum irred res}.

\begin{prop}\label{prop:fixed-pt CP(V)}
    For $H\leq G$ with $|H|=k$ and notation \eqref{eq:sum irred res}
    $$\C P(V)^H\simeq\bigsqcup_{i=1}^{k}\C P^{n_i-1}.$$
\end{prop}
\begin{proof}
Use the presentation \eqref{eq:sum irred res} and write 
   $$\C P(V)^H=\C P(i_H^*V)^H=\{L\in\C P(V) ~|~ h\cdot L=L, \forall h\in H\}.$$
Since the lines $L$ are $H$-invariant subspaces, by the irreducibility of $V$ they must be contained in the one-dimensional representations of $i_H^*V$. Then
\begin{myeq}\label{eq:fixed-pt rep}
    \C P(i_H^*V)^H\simeq\C P(n_1\phi_1+n_2\phi_2+\cdots+n_k\phi_k)^H.
\end{myeq}However a line $L\subset V$ is fixed by a subgroup $H\leq G$ if and only if $H$ caries $L$ to itself. It implies that $L$ lies entirely in one of the isotypical components of $i_H^*V$, which is,
   $$L\subset n_i\phi_i \text{ for some } i.$$

Therefore we have
   \begin{align*}
       \C P(V)^H
       &\simeq\C P(n_1\phi_1+n_2\phi_2+\cdots+n_k\phi_k)^H\\
       &\simeq\C P(\C^{n_1})\sqcup\C P(\C^{n_2})\sqcup\cdots\sqcup\C P(\C^{n_k})\\
       &\simeq\bigsqcup_{i=1}^{k}\C P^{n_i-1}
   \end{align*}
   and the claim follows.
\end{proof}

\begin{exam}[Abelian group]
Let $G$ be an Abelian group with order $n$. Then we can write a $G$-representation $V$ as a direct sum of one-dimensional $\phi_i$, that is, 
$$V=\sum_{i=0}^{n-1}n_i\phi_i.$$
As all $G$-invariant complex lines are exactly the one-dimensional representations of $V$, 
$$\C P(V)^G\simeq\bigsqcup_{i=0}^{n-1}\C P^{n_i-1}.$$
\end{exam}
\vspace{1cm}

Fix a subgroup $H\leq G$. As in \eqref{eq:sum irred res} we write
    $$i_H^*\psi_{i+1}=\sum_{j=1}^{k}n_j\phi_j\oplus \sum_{r=1}^{l}m_r\psi_r$$
such that $\phi_j$ are one-dimensional $H$-representations, $\psi_r$ are $H$-irreducible representations with $\dim(\psi_r)>1$, and $n_j, m_r>0$. Then by Proposition \ref{prop:fixed-pt CP(V)} we obtain
\begin{myeq}\label{eq:split CP(psi_{i+1}}
    \C P(\psi_{i+1})^H\simeq \C P(i_H^*\psi_{i+1})^H\simeq\bigsqcup_{j=1}^{k}\C P^{n_{j}-1}. 
\end{myeq}

Apply the fixed point $\Phi^H(-)$ on \eqref{eq:line bundle gamma psi} and write 
\begin{myeq}\label{eq:fixed line bundle}
    S(i_H^*\psi_{i+1})\times_\T\C\xrightarrow[]{\gamma_{n_j}}\C P(n_j\phi_j)\subseteq\C P(i_H^*\psi_{i+1})
\end{myeq}for the subbundle of $\gamma_{\psi_{i+1}}$ over a component of the $H$-fixed points. We compute the space level $H$-fixed points of  $\gamma_{n_i}$ fiberwise. An element $x\in D(V_i\otimes\gamma_{n_i})$ is $H$-fixed if $\gamma_{\psi_{i+1}}(x)\in\C P(\psi_{i+1})$ is $H$-fixed and $x$ is $H$-fixed in $\gamma_{\psi_{i+1}}^{-1}\gamma_{\psi_{i+1}}(x)$ under the induced $H$-action. More explicitly we view this $H$-fixed action via a pullback bundle
\[\begin{tikzcd}
	{i_H^*V_i\otimes\gamma_{n_i}} & {V_i\otimes\gamma_{\psi_{i+1}}} \\
	{\mathbb{C}P^{n_i-1}\cong\mathbb{C}P(n_i\phi_i)} & {\mathbb{C}P(\psi_{i+1})}
	\arrow[from=1-1, to=1-2]
	\arrow[from=1-1, to=2-1]
	\arrow[from=1-2, to=2-2]
	\arrow["{H-\text{map}}"', from=2-1, to=2-2]
\end{tikzcd}\]
From the $H$-map in the above diagram it follows that all elements in $D(V_i\otimes\gamma_{\psi_{i+1}})^H$ are the families of $H$-fixed points 
\begin{myeq}\label{eq:fixed disc bundle}
   \{x\in D(V_i\otimes\gamma_{\psi_{i+1}})^H ~|~ \gamma_{\psi_{i+1}}(x)\in\C P^{n_i-1}\}=D(i_H^*V_i\otimes\gamma_{n_i})^H=D(V_i^H\otimes\gamma_{n_i}). 
\end{myeq}

With this ingredient, we are ready to state the following decomposition.
\begin{prop}\label{prop:fixed pt thom space}
For $H\leq G$ such that $|H|=k$ with notations \eqref{eq:sum gen irred} and \eqref{eq:fixed line bundle}
    $$\text{Th}(V_i\otimes\gamma_{\psi_{i+1}})^H\simeq\bigvee_{j=1}^{k}\text{Th}(V_i^H\otimes\gamma_{n_j})$$
    where $\gamma_{n_j}$ is defined in \eqref{eq:fixed line bundle}.
\end{prop}
\begin{proof}
Choose a $G$-invariant inner product on $V_i\otimes\gamma_{\psi_{i+1}}$ and write
$$\text{Th}(V_i\otimes\gamma_{\psi_{i+1}})\cong D(V_i\otimes\gamma_{\psi_{i+1}})/S(V_i\otimes\gamma_{\psi_{i+1}}).$$
Now we compute the space level $H$-fixed points. First we observe that  
$$\text{Th}(V_i\otimes\gamma_{\psi_{i+1}})^H\simeq (D(V_i\otimes\gamma_{\psi_{i+1}})^H\setminus S(V_i\otimes\gamma_{\psi_{i+1}})^H)^+$$
where $(-)^+$ denotes the one-point compactification.

By \eqref{eq:fixed disc bundle} we get
$$D(V_i\otimes\gamma_{\psi_{i+1}})^H\simeq\bigsqcup_{j=1}^kD(V_i\otimes{\gamma_{\psi_{i+1}}}_{|_{\C P^{n_j-1}}})^H\simeq\bigsqcup_{j=1}^kD(V_i^H\otimes\gamma_{n_j}).$$

Following the presentations \eqref{eq:fixed line bundle} and \eqref{eq:fixed disc bundle},
\begin{align*}
    \text{Th}(V_i\otimes\gamma_{\psi_{i+1}})^H
    &\simeq(\bigsqcup_{j=1}^{k}(D(V_i^H\otimes\gamma_{n_j})\setminus S(V_i^H\otimes\gamma_{n_j})))^+\\
    &\simeq\bigvee_{j=1}^{k}\text{Th}(V_i^H\otimes\gamma_{n_j}).
\end{align*}
\end{proof}

 By applying the Thom isomorphism \cite[10.2]{MS74} to Proposition \ref{prop:fixed pt thom space}, we immediately obtain the following corollary.
\begin{cor}\label{cor:homology thom space}
   The rational homology groups
    $$H_*(\text{Th}(V_i\otimes\gamma_{\psi_{i+1}})^H;\Q)$$
    are concentrated in even degrees for all subgroups $H\leq G$. 
\end{cor}
\begin{proof}
    Use Proposition \ref{prop:fixed pt thom space} and apply the Thom isomorphism
    $$\tilde{H}_*(\text{Th}(V_i\otimes\gamma_{\psi_{i+1}})^H;\Q)\cong H_{*-2d}(\C P(\psi_{i+1})^H;\Q)$$
    where $d=\text{dim}V_i$. We claim that the right term lies within even degrees. 
    
    From the decomposition of $\C P(\psi_{i+1})$ in \eqref{eq:split CP(psi_{i+1}}
    $$H_{*-2d}(\C P(\psi_{i+1})^H;\Q)\cong H_{*-2d}(\bigsqcup_{j=1}^{k}\C P^{n_j-1};\Q)\cong\bigoplus_{j=1}^{k}H_{*-2d}(\C P^{n_j-1};\Q)$$
    which are concentrated in even degrees. Then it simply proves the claim.
\end{proof}

As a generalization to Theorem \ref{splitting CP(V)}, we now have the decomposition of the $\C P(V)$ for $V$ a complex representation of any finite group $G$. 
\begin{thm}\label{thm:split thom space}
    There is an equivalence
    $$\Sigma^\infty\C P(V)_+\simeq_\Q \bigvee_{i=0}^{n-1}\text{Th}(V_i\otimes\gamma_{\psi_{i+1}})$$
    with notations $V, V_i$ in \eqref{eq:sum gen irred} and $\gamma_{\psi_{i+1}}$ in \eqref{eq:line bundle gamma psi}.
\end{thm}
\begin{proof} 
We prove by induction on $n=\text{dim}V$. Recall the notation \eqref{eq:sum gen irred} and the cofiber sequence \eqref{eq:thom cof seq 2}
  $$\C P(V_{i})_+\to\C P(V_{i+1})_+\to \text{Th}(V_{i}\otimes\gamma_{\psi_{i+1}}).$$
Inductively we assume that
 $$\C P(V_i)_+\simeq_\Q\C P(V_{i-1})_+ \bigvee \text{Th}(V_{i-1}\otimes\gamma_{\psi_{i}}).$$
It is enough to show that the connecting map
 $$\text{Th}(V_{n-1}\otimes\gamma_{\psi_{n}})\xrightarrow[]{\delta_n}\Sigma\C P(V_{n-1})_+$$
 is the zero map. 

 By the induction step for $\C P(V_i)_+$, we have
 $$\C P(V_{n-1})_+\simeq_\Q \C P(V_{n-2})_+\bigvee \text{Th}(V_{n-2}\otimes\gamma_{\psi_{n-1}}).$$
 We then prove for $i=n$ as follows. Consider the connecting map
 $$\text{Th}(V_{n-1}\otimes\gamma_{\psi_{n}})\xrightarrow[]{\delta_n}\Sigma\C P(V_{n-1})_+\simeq_\Q\Sigma\C P(V_{n-2})_+\bigvee\Sigma\text{Th}(V_{n-2}\otimes\gamma_{\psi_{n-1}})$$
 and apply the geometric fixed point functor $\Phi^H(-)$
  $$\text{Th}(V_{n-1}\otimes\gamma_{\psi_{n}})^H\xrightarrow[]{\Phi^H(\delta_n)}\Sigma\C P(V_{n-2})_+^H\bigvee\Sigma\text{Th}(V_{n-2}\otimes\gamma_{\psi_{n-1}})^H.$$
  
  Following Proposition \ref{prop:fixed-pt CP(V)} and Corollary \ref{cor:homology thom space}, the rational homology groups of the left term are concentrated in even degrees, while those of the right term are concentrated in odd degrees. This shows that the connecting map is the zero map by \eqref{eq:id alg model}, and the result follows. 
\end{proof}

\section{Rational stable homotopy type of $Gr_n(V)$ over general groups}\label{sec:gr_n(v)}
The purpose of this section is to show that the Grassmannian $Gr_n(V)$ rationally splits as a smash product of Thom spaces over lower Grassmannians on irreducible representations in $V$ (Theorem \ref{thm:general split gr}). First we introduce some constructions, and then deduce the necessary elements for proving the main result.

Suppose $V$ is an $m$-dimensional representation equipped with a $G$-invariant inner product. Express $V$ as the orthogonal direct sum $V=V'\oplus V_0$, so that $V'=V-V_0$. We then form the following definition.
\begin{defn}\label{def:dim 1 gr}
     Let $V_0$ be a subspace of $V$. Define 
     $$Gr_n(V)_{V_{0,1}}=\{W\in Gr_n(V) ~|~ \text{dim}(W\cap V_0)\geq 1\}.$$
 \end{defn}
By this Definition \ref{def:dim 1 gr} it follows that
\begin{myeq} \label{eq:seq dim 1 gr}
  Gr_n(V)_{V_{0,1}+}\to Gr_n(V)_+\to Gr_n(V)/Gr_n(V)_{V_{0,1}}
\end{myeq}is a cofiber sequence. The above quotient is described more explicitly as follows. Consider a projection map $p$ in a commutative triangle
\[\begin{tikzcd}
	W & V \\
	& {V'}
	\arrow[hook, from=1-1, to=1-2]
	\arrow[hook', from=1-1, to=2-2]
	\arrow["p", from=1-2, to=2-2]
\end{tikzcd}\]
and define an associated map
$$Gr_n(V)\setminus Gr_n(V)_{V_{0,1}}\to Gr_n(V'), W\mapsto p(W)$$
where $W$ is identified as in Definition \ref{def:dim 1 gr}. It shows that the space $Gr_n(V)\setminus Gr_n(V)_{V_{0,1}}$ naturally corresponds to a vector bundle
\begin{myeq}\label{eq:n bundle gr}
   E(V_0\otimes\xi_n)\xrightarrow[]{ V_0\otimes\xi_n} Gr_n(V')
\end{myeq}
where $\xi_n$ is the canonical $n$-bundle over $Gr_n(V')$. Then it gives a $G$-homeomorphism
$$Gr_n(V)\setminus Gr_n(V)_{V_{0,1}}\cong E(V_0\otimes\xi_n)$$
and induces homotopy equivalences
$$Gr_n(V)/Gr_n(V)_{V_{0,1}}\simeq (E(V_0\otimes\xi_n))^+\simeq\text{Th}(Gr_n(V'),V_0\otimes\xi_n).$$
Using this information the cofiber sequence \eqref{eq:seq dim 1 gr} becomes
\begin{myeq} \label{eq:sim seq dim 1 gr}
  Gr_n(V)_{V_{0,1}+}\to Gr_n(V)_+\to \text{Th}(Gr_n(V'),V_0\otimes\xi_n).
\end{myeq}

In the same terminology as in \S\ref{sec:cp(v)} we claim that this sequence \eqref{eq:sim seq dim 1 gr} splits in $\Spec_\Q^G$ and gives the stable splitting of $Gr_n(V)$ in Theorem \ref{thm:general split gr}. To provide a suitable description in the context of $Gr_n(V)$, we need to rewrite some of the previous expressions.

First rewrite the restriction in \eqref{eq:sum irred res} for $V$ as
\begin{myeq}\label{eq:res rep v}
    i_H^*V=\sum_{i=1}^{k}n_i\psi_i
\end{myeq}where $\psi_i$ are irreducible $H$-representations and $n_i>0$. Using this expression we obtain the following consequence. 
\begin{prop}\label{prop:fixed-pt gr}
    For $H\leq G$ with $|H|=k$ and expression \eqref{eq:res rep v}
     $$Gr_n(V)^H\simeq \bigsqcup_{\sum\limits_{i=1}^{k}m_i\text{dim}\psi_i=n} \prod_{i=1}^{k}Gr_{m_i}(\C^{n_i}).$$
\end{prop}
\begin{proof}
First notice that
$$Gr_n(V)^H=\{W\in Gr_n(V) ~|~ h\cdot W=W, \forall h\in H\}.$$
From this we see that every $n$-plane $W$ is an $H$-invariant subspace of $V$, so $W$ itself must splits into subspaces of $H$-isotypical components. In precise terms,
\begin{align*}
\{W\subset n_i\psi_i ~|~ \text{dim}W=m_i\text{dim}\psi_i, W \text{ is } H\text{-fixed}\}
&=\text{Hom}_H(m_i\psi_i,n_i\psi_i)/\text{Hom}_H(m_i\psi_i,m_i\psi_i)\\
&=\{\C^{m_i}\to\C^{n_i}\text{ injective }\}/GL_{m_i}(\C)\\
&=Gr_{m_i}(\C^{n_i}).
\end{align*}

If we write 
   $$ \underline{m}=(m_1,\cdots,m_k)\text{ such that } \sum_{i=1}^{k}m_i\text{dim}\psi_i=n,$$
\begin{myeq}\label{eq:not gr m,h}
    G_{\underline{m},H}=\{W\in Gr_n(V)^H~|~(W\cap n_i\psi_i)=m_i\psi_i\}\cong\prod_{i=1}^{k}Gr_{m_i}(\C^{n_i}).
\end{myeq}
From this it follows that
   \begin{align*}
       Gr_n(V)^H
       &\simeq\bigsqcup_{\sum\limits_{i=1}^{k}m_i\text{dim}\psi_i=n} Gr_{m_1}(\C^{n_1})\times\cdots\times Gr_{m_k}(\C^{n_k})\\
       &\simeq\bigsqcup_{\sum\limits_{i=1}^{k}m_i\text{dim}\psi_i=n}\prod_{i=1}^{k}Gr_{m_i}(\C^{n_i}).
   \end{align*}
\end{proof}
Now we compute the $H$-fixed points of the space $\text{Th}(Gr_n(V),U\otimes\xi_n)$ where $U$ is a complex $G$-representation. We denote by $\xi_{\underline{m},H}$ the restriction of $\xi_n$ to $Gr_{\underline{m},H}$.

\begin{prop}\label{prop:fixed thom n gr}
    For a subgroup $H\leq G$ and $U$ a complex $G$-representation 
    $$\text{Th}(Gr_n(V), U\otimes\xi_n)^H\simeq\bigvee_{\underline{m}}\text{Th}(G_{\underline{m},H}(V)^H, U^H\otimes\xi_{\underline{m},H}).$$
    It follows that $H_*(\text{Th}(Gr_n(V), U\otimes\xi_n)^H;\Q)$ are concentrated in even degrees.
\end{prop}
\begin{proof}
Under the assumption that $U\otimes\xi_{n}$ has a $G$-invariant inner product,
$$\text{Th}(Gr_n(V), U\otimes\xi_{n})\cong D(Gr_n(V), U\otimes\xi_{n})/S(Gr_n(V), U\otimes\xi_{n}).$$
We then derive the space level $H$-fixed points 
\begin{align*}
    \text{Th}(Gr_n(V), U\otimes\xi_{n})^H
    &\simeq (D(Gr_n(V), U\otimes\xi_{n})^H\setminus S(Gr_n(V), U\otimes\xi_{n})^H)^+\\
    &\simeq(\bigsqcup_{\underline{m}}(D(G_{\underline{m},H}(V)^H, U^H\otimes\xi_{\underline{m},H})\setminus S(G_{\underline{m},H}(V)^H, U^H\otimes\xi_{\underline{m},H})))^+\\
     &\simeq\bigvee_{\underline{m}}\text{Th}(G_{\underline{m},H}(V')^H, U^H\otimes\xi_{\underline{m},H}).
\end{align*}

Moreover the space $G_{\underline{m},H}$ is a product of lower Grassmannians in \eqref{eq:not gr m,h}, and thus it has a CW complex structure with even-dimensional cells. Therefore the second statement follows from the Thom isomorphism. 
\end{proof}

We now prove our main result by using a filtered cofiber sequence of \eqref{eq:sim seq dim 1 gr}. For this we rewrite Definition \ref{def:dim 1 gr} in a more general form as follows. 
 \begin{defn}\label{def:dim k gr}
     Let $V_0$ be a subspace of $V$. Define 
     $$Gr_n(V)_{V_{0,k}}=\{W\in Gr_n(V) ~|~ \text{dim}(W\cap V_0)\geq k\}.$$
 \end{defn}
This formulation yields a complete filtration
  $$Gr_n(V_0)= Gr_n(V)_{V_{0,n}}\subseteq\cdots\subseteq Gr_n(V)_{V_{0,1}}\subseteq Gr_n(V)_{V_{0,0}}=Gr_n(V)$$
and a cofiber sequence
\begin{myeq} \label{eq:seq dim k gr}
  Gr_n(V)_{V_{0,k+1}+}\to Gr_n(V)_{V_{0,k}+}\to Gr_n(V)_{V_{0,k}}/Gr_n(V)_{V_{0,k+1}}. 
\end{myeq}

In the same spirit as before, consider a commutative triangle
\[\begin{tikzcd}
	W & V \\
	& {V'}
	\arrow[hook, from=1-1, to=1-2]
	\arrow[hook', from=1-1, to=2-2]
	\arrow["p", from=1-2, to=2-2]
\end{tikzcd}\]
so that $\text{dim}(W\cap V_0)=k-1$ in $W$ and $\text{dim}(W\cap V_0)=n-k+1$ in $V'$. Define a map
$$Gr_n(V)_{V_{0,k}}\setminus Gr_n(V)_{V_{0,k+1}}\to Gr_{n-k}(V')\times Gr_k(V_0), W\mapsto (p(W),W\cap V_0)$$
where $W$ equals $(W\cap V_0)\oplus W'$, with $W'$ as the orthogonal complement of $W\cap V_0$. Then for every $W'\in Gr_{n-k}(V)$, satisfying $p(W')=p(W)$, maps isomorphically onto $p(W)$ in $Gr_{n-k}(V')$. Therefore, $W'$ is a graph of a linear map from $p(W)$ to $V_0$. This shows that it canonically associates with $W'$ as an element of the vector bundle 
\begin{myeq}\label{eq:n-k bundle gr}
    E(V_0\otimes\xi_{n-k})\xrightarrow[]{V_0\otimes\xi_{n-k}}Gr_{n-k}(V')
\end{myeq}where $\xi_{n-k}$ is the canonical $(n-k)$-bundle over $Gr_{n-k}(V')$. If $k=0$, then $\xi_{n}$ is the canonical $n$-bundle over $Gr_n(V')$ as in \eqref{eq:n bundle gr}. Then we get a $G$-homeomorphism
$$ Gr_n(V)_{V_{0,k}}\setminus Gr_n(V)_{V_{0,k+1}}\cong E(V_0\otimes\xi_{n-k})\times Gr_k(V_0)$$
and homotopy equivalences
\begin{align*}
    Gr_n(V)_{V_{0,k}}/ Gr_n(V)_{V_{0,k+1}}
    &\simeq (E(V_0\otimes\xi_{n-k})\times Gr_k(V_0))^+\\
    &\simeq\text{Th}(Gr_{n-k}(V'),V_0\otimes\xi_{n-k})\land Gr_k(V_0)_+.
\end{align*}
From this we simplify the cofiber sequence \eqref{eq:seq dim k gr} to
\begin{myeq} \label{eq:sim seq dim n k gr}
  Gr_n(V)_{V_{0,k+1}+}\to  Gr_n(V)_{V_{0,k}+}\to \text{Th}(Gr_{n-k}(V'),V_0\otimes\xi_{n-k})\land Gr_k(V_0)_+.
\end{myeq}

Under the above construction we produce the stable decomposition of $Gr_n(V)$.
\begin{thm}\label{thm:general split gr}
    There is a stable homotopy equivalence
    $$\Sigma^\infty Gr_n(V)_+\simeq_\Q\bigvee_{k=1}^{n}\text{Th}(Gr_k(V'), V_0\otimes\xi_k)\land Gr_{n-k}(V_0)_+\bigvee Gr_n(V_0)_+$$
    where $\xi_k$ is defined as in \eqref{eq:n-k bundle gr}.
\end{thm}
\begin{proof} 
We prove by induction on indexing $k$ that the cofiber sequence \eqref{eq:sim seq dim n k gr}
  $$ Gr_n(V)_{V_{0,k+1}+}\to Gr_n(V)_{V_{0,k}+}\to \text{Th}(Gr_{n-k}(V'),V_0\otimes\xi_{n-k})\land Gr_{k}(V_0)_+$$
  splits in $\Spec_\Q^G$.
In the first step we have the cofiber sequence \eqref{eq:sim seq dim 1 gr}
$$Gr_n(V)_{V_{0,1}+}\to Gr_n(V)_+\to \text{Th}(Gr_n(V'),V_0\otimes\xi_n)$$  
and its $H$-fixed connecting map
$$\text{Th}(Gr_n(V'),V_0\otimes\xi_n)^H\xrightarrow[]{\delta_n^H}\Sigma (Gr_n(V)_{V_{0,1}})_+^H.$$
Following Propositions \ref{prop:fixed-pt gr} and \ref{prop:fixed thom n gr}, the rational homology groups of the right side are concentrated in odd degrees, while those of the left side are concentrated in even degrees, by induction. This shows that the connecting map is the zero map by \eqref{eq:id alg model}. Then 
\begin{myeq}\label{eq:split gr proof}
    Gr_n(V)_+\simeq_\Q Gr_n(V)_{V_{0,1}+}\bigvee \text{Th}(Gr_n(V'),V_0\otimes\xi_n).
\end{myeq}

However from the inductive step of \eqref{eq:sim seq dim n k gr}, we have the lower decompositions
\begin{align*}
\begin{split}
   Gr_n(V)_{V_{0,1}+}\simeq_\Q {}& Gr_n(V)_{V_{0,2}+}\bigvee \text{Th}(Gr_{n-1}(V'), V_{0}\otimes\xi_{n-1})\land Gr_{1}(V_0)_+
\end{split}\\
\begin{split}
    Gr_n(V)_{V_{0,2}+}\simeq_\Q {}& Gr_n(V)_{V_{0,3}+}\bigvee \text{Th}(Gr_{n-2}(V'), V_{0}\otimes\xi_{n-2})\land Gr_2(V_0)_+ 
    \end{split}\\
  \begin{split}  
    \vdots
\end{split}\\
      Gr_n(V)_{V_{0,n-1}+}\simeq_\Q {}& Gr_n(V)_{V_{0,n}+}\bigvee \text{Th}(Gr_{1}(V'), V_{0}\otimes\xi_{1})\land Gr_{n-1}(V_0)_+.
\end{align*}
The proof is completed by combining all the lower decompositions of  $Gr_n(V)_{V_{0,k}}$ to yield $Gr_n(V)_{V_{0,1}}$ and applying \eqref{eq:split gr proof}.
\end{proof}

\noindent

\bibliographystyle{siam}
\bibliography{REGr}
\end{document}